\def\nnewpage{}

\documentclass[10pt]{article}

 \usepackage{amsmath,amsthm,amsfonts,amssymb,url}
 \usepackage[numbers]{natbib}
 \usepackage{graphicx}
 \usepackage{color}
 \usepackage{enumitem}

\newtheorem{Thm}{Theorem}
\newtheorem{Lem}{Lemma}

\newtheorem{Coro}{Corollary}

\theoremstyle{definition}
\newtheorem{Algo}{Algorithm}
\newtheorem{Rem}{Remark}

\newtheorem{Example}{Example}

\newcommand{\comment}[1]{}

\def\indd#1{{\bf 1}_{\{#1\}}}
\newcommand{\proba}{\mathbb P}
\newcommand{\esp}{{\mathbb E}}

\newcommand{\inv}{^{-1}}

\newcommand{\var}{{\rm{Var}}}

\newcommand{\calB}{{\cal B}}

\newcommand{\filF}{{\cal F}}

\newcommand{\calL}{{\cal L}}

\newcommand{\calX}{{\cal X}}

\def\B{{\mathbb B}}

\def\G{{\mathbb G}}

\def\indn#1{\{#1_n\}_{n\in \mathbb N}}

\newcommand{\eqnh}{\begin{eqnarray*}}
\newcommand{\eqne}{\end{eqnarray*}}
\newcommand{\eqnhn}{\begin{eqnarray}}
\newcommand{\eqnen}{\end{eqnarray}}
\newcommand{\equh}{\begin{equation}}
\newcommand{\eque}{\end{equation}}

\def\summ#1#2#3{\sum_{#1 = #2}^{#3}}

\def\sif#1#2{\sum_{#1=#2}^\infty}

\newcommand{\widebar}{\overline}

\def\topp#1{^{(#1)}}

\def\nnTV#1{\left\|#1\right\|_{\rm TV}}
\def\snnTV#1{\|#1\|_{\rm TV}}

\def\abs#1{\left|#1\right|}

\def\ccbb#1{\left\{#1\right\}}

\def\spp#1{(#1)}
\def\pp#1{\left(#1\right)} 
\def\bbpp#1{\left(#1\right)} 
\def\bb#1{\left[#1\right]}
\def\sbb#1{[#1]}

\def\floor#1{\left\lfloor #1 \right\rfloor}

\def\d{{\rm d}}

\def\B{{\mathbb B}}

\def\mand{\mbox{ and }}
\def\qmand{\quad\mbox{ and }\quad}

\def\mwith{\mbox{ with }}

\def\wt#1{\widetilde{#1}}
\def\wb#1{\widebar{#1}}

\def\what#1{\widehat{#1}}

\def\weakto{\Rightarrow}

\def\R{{\mathbb R}}

\def\N{{\mathbb N}}

\def\rset{\mathbb R}

\newcommand{\eqdef}{\ensuremath{\stackrel{\mathrm{def}}{=}}}

\def\F{\mathcal{F}} 
\def\B{\mathcal{B}} 

\def\G{\mathbb{G}}

\def\PP{\mathbb{P}} 
\def\PE{\mathbb{E}} 

\usepackage{bm}

\newcounter{hypoconbis}
\newcounter{saveconbis}
\newcommand\debutA{
\begin{list} {\textbf{A\arabic{hypoconbis}}}{\usecounter{hypoconbis}}\setcounter{hypoconbis}{\value{saveconbis}}
}
\newcommand\finA{\end{list}\setcounter{saveconbis}{\value{hypoconbis}}}

\newlist{conditions}{enumerate}{1}
\setlist[conditions]{label={\bf A\arabic*},ref=(A\arabic*)}

\newcounter{AC}
\newcommand\assumpH
{
\addtocounter{AC}{1}
\begin{conditions}[start = \value{AC}]
}
\newcommand\assumpE
{
\setcounter{AC}{\value{conditionsi}}
\end{conditions}
}

\begin{document}\sloppy

\title{On the Convergence Rates of Some Adaptive Markov Chain Monte Carlo Algorithms}
\author{Yves Atchad\'e\thanks{
 Department of Statistics, the University of Michigan,
 439 West Hall, 1085 S.\ University, Ann Arbor, MI 48109--1107, USA.
  {\em Email:}
  \texttt{yvesa@umich.edu}.
  }\ \ and Yizao Wang\thanks{
  Department of Mathematical Sciences,
University of Cincinnati,
2815 Commons Way, ML$\#0025$,
Cincinnati, OH, 45221--0025, USA.
  {\em Email:}
  \texttt{yizao.wang@uc.edu}.
  }}
\maketitle

\begin{abstract}
This paper studies the mixing time of certain adaptive Markov Chain Monte Carlo algorithms. Under some regularity conditions, we show that the convergence rate of Importance Resampling MCMC algorithm, measured in terms of  the total variation distance is $O(n^{-1})$. By means of an example, we establish that in general, this algorithm does not converge at a faster rate.  We also study the interacting tempering algorithm, a simplified version of Equi-Energy sampler, and establish that its mixing time is of order $O(n^{-1/2})$. \medskip


\noindent{\it MSC2010 subject classifications:} Primary 65C05, 65C40; Secondary 60J05.
\end{abstract}

\section{Introduction}

Constructing Markov Chain Monte Carlo (MCMC) transition kernels to sample efficiently from a given distribution $\pi$, say, is a difficult task in practice, as it  requires a careful choice and tuning of the  kernel. 
The development of adaptive MCMC (AMCMC) methods is partly motivated by the need of overcoming  this difficulty. Instead of having a fixed Markov kernel $P$, at each round $n$ an AMCMC algorithm selects a kernel $P_{\what\theta_n}$ from a family of Markov kernels $\{P_\theta\}_{\theta\in\Theta}$, where the value (parameter) $\what\theta_n$ is computed based on possibly all the samples generated up to time $n$, so that the transition kernel is automatically self-adapted. 
See for example the recent survey by \citet{atchade11adaptive} and the references therein.

In this paper, we investigate the convergence rates of two AMCMC algorithms: the {\it Importance Resampling MCMC (IRMCMC)}  algorithm introduced by \citet{atchade09resampling}, and the {\it Equi-Energy (EE) sampler} by \citet{kou06equi}. The IRMCMC algorithm is also referred to {\it interacting annealing} algorithm \citep{bercu12fluctuations}. For the EE sampler, we actually focus on a simplified version, which is sometimes referred to as {\it interacting tempering} (IT) algorithm \citep{fort14central}. 

Throughout the paper we denote by $\{X_n\}_{n\in\N}$ the random process generated by either of these algorithms. Limit theorems, notably convergence of marginal distributions and law of large numbers have been known. See for example  \citep{andrieu08note,andrieu11nonlinear,atchade09resampling,atchade10cautionary,fort11convergence}, among others. 
Central limit theorems for such AMCMC algorithms have only been considered recently by \citet{fort14central} and \citet{bercu12fluctuations}. In short, introducing the auxiliary chain makes the stochastic process no longer Markov, which raises considerable technical difficulties. 

In this paper, we study the {\it convergence rate} (or mixing time) of the IRMCMC and IT algorithms. That is, we provide upper bounds on the distances between $\calL_{X_n}$ (the distribution of $X_n$) and the target distribution. Such mixing time results provide information on the burn-in time of the algorithm.  
Few results in the literature are known on the mixing rates of AMCMC. \citet{andrieu07efficiency}  considered AMCMC with a finite-dimensional parameter. Related results have been obtained by \citet{woodard09conditions,schmidler11lower} on convergence rates of AMCMC and related algorithms, although with a different point of view from us: they focused on the lower bound in terms of the problem size, not the simulation rounds.  

We show that the IRMCMC algorithm has convergence rate of order $O(n\inv)$. In particular, we also provide a simple example, for which the convergence rate has lower bound $1/n$. We also show that for $m$-tuple IRMCMC algorithm (to be defined in section~\ref{sec:mIRMCMC}), the mixing time is within $O(n\inv(\log n)^{m-1})$.
For the IT algorithm, under some regularity conditions, we show that the rate of convergence is $O(n^{-1/2})$ in terms of a slightly weaker norm than the total variation distance.  
These results do not automatically lead to a precise method for selecting burn-in periods, because the constants in the derived bounds are hard to compute in most practical cases. However, from a practical viewpoint, this analysis can be viewed as a cautionary tale, suggesting that AMCMC samplers based on auxiliary chains typically requires longer burn-in periods than standard, well-behaved MCMC samplers.

The rest of the paper is organized as follows. The remaining of the introduction gives a general description of the algorithms considered in the paper and introduces some notation. Section~\ref{sec:IRMCMC} is devoted to IRMCMC algorithm. The convergence rate is established in Section~\ref{sec:IRMCMCrate}, and for multiple IRMCMC algorithm in Section~\ref{sec:mIRMCMC}. Section~\ref{sec:EE} is devoted to IT algorithm.

\subsection{Notation}\label{sec:notation}
We assume that the state space $\calX$ is a Polish space equipped with a metric $\textsf{d}$, and $\calB$ is the associated Borel $\sigma$-algebra. In addition, $(\calX,\B)$ is a measure space with a reference $\sigma$-finite measure, which we denote for short by $\d x$. Let $\pi$ and $\pi_Y$ be  probability measures on $(\calX,\B)$. 
We assume that $\pi$ and $\pi_Y$ are both absolutely continuous with respect to $\d x$ and with a little abuse of notation, we also use $\pi$ and $\pi_Y$ to denote the density respectively. That is, we write $\pi(\d x) = \pi(x)\d x$ and similarly for $\pi_Y$. For a transition kernel $Q$, a measure $\nu$ and a function $h$, we shall write $\nu Q(\cdot)\eqdef \int\nu(\d z)Q(z,\cdot)$, and $Qh(\cdot)\eqdef \int Q(\cdot,\d z)h(z)$.

In this paper, an AMCMC algorithm  is a stochastic process $\{(X_n,Y_n)\}_{n\geq 0}$ in $\calX\times\calX$, designed such that the main chain $X_n$ converges to the target distribution $\pi$ in a certain sense to be described precisely later. We also assume that the auxiliary chain $\{Y_n\}_{n\geq 0}$ converges to $\pi_Y$. 
For the two algorithms analyzed in this paper, we assume that the evolution of the auxiliary chain is independent of the main chain. The auxiliary chain is not necessarily Markov.  Write $\filF_n = \sigma(X_0,\dots,X_n,Y_0,\dots,Y_n)$. 

We denote $\what\pi_{Y,n}$ the empirical measure associated to  the auxiliary chain $\indn Y$ defined by 
$\what \pi_{Y,n}(\cdot) \eqdef \frac1n\summ i1n\delta_{Y_i}(\cdot)$.
For functions $f:\calX\to\R$, we write 
\[
\what\pi_{Y,n}(\wb f) \eqdef \what\pi_{Y,n}(f) - \pi_Y(f).
\]
We avoid writing $\wb f$ for the centered version of $f$, as it would be  unclear with respect to which measure $f$ is centered, especially in the setup of multiple chains. 
We let $C$ denote general constants that do not depend on $n$, but may change from line to line. \bigskip

\noindent {\bf Acknowledgment} We thank an anonymous referee for the very careful reading of and helpful suggestions for our paper.

\nnewpage
\section{Importance Resampling MCMC}\label{sec:IRMCMC}
We consider the  {importance-resampling Markov Chain Monte Carlo} method described in \citet{atchade09resampling}. 

\begin{Algo}[IRMCMC]\label{algo:IRMCMC}
Fix $\epsilon\in(0,1)$. Pick arbitrary $X_0 = x_0$ and $Y_0 = y_0$. Let $P$ be an arbitrary Markov kernel with invariant distribution $\pi$. At each round $n$,  $X_n$ and $Y_n$ are conditionally independent given $\F_{n-1}$, and  
\[
X_{n}\mid \F_{n-1} \sim \left\{
\begin{array}{l@{\mbox{ w.p.~}}l}
P(X_{n-1},\cdot) & 1-\epsilon\,,\\
\what\theta_{n-1}(\cdot) &  \epsilon\,,
\end{array}\right.
\]
where $\what\theta_n$ is the (randomly) weighted empirical distribution defined by
\equh\label{eq:theta}
\what\theta_n(\cdot) = \summ i1{n}\frac{\wt w(Y_i)}{\summ j1{n}\wt w(Y_j)}\delta_{Y_i}(\cdot)=\frac{\int_\cdot \wt w(z)\what\pi_{Y,n}(\d z)}{\int_{\calX} \wt w(z)\what\pi_{Y,n}(\d z)},
\eque
with $\wt w(y) \propto \pi(y)/\pi_Y(y)=: w(y)$, and $\what\theta_0 = \delta_{y_0}$. Recall that $\pi_Y$ is the limiting distribution of the auxiliary chain $\{Y_n\}_{n\geq 0}$. We assume $|w|_\infty\eqdef\sup_{x\in\calX}|w(x)|<\infty$. 
\end{Algo}
For all probability measures $\theta$ on $\calX$, we introduce
\equh\label{eq:Ptheta}
P_\theta(x,\cdot) = (1-\epsilon)P(x,\cdot) + \epsilon \theta(\cdot)\,.
\eque
 In this way, for any bounded function $f:\calX\to\R$, 
$\esp (f(X_{n+1})\mid \filF_n) = P_{\what\theta_n}f(X_n)$ almost surely.
 \begin{Rem}\label{rem:w}
The assumption on the boundedness of $w$ is not too restrict. Indeed, very often in practice, we have $\wt\pi$, the un-normalized density function of $\pi$ as a bounded function, and set the auxiliary chain with stationary distribution $\wt\pi_Y\propto\pi_Y$ obtained by $\wt\pi_Y = \wt\pi^T$ with $T\in(0,1)$. In this case, $\wt w = \wt \pi/\wt \pi_Y$ is bounded and thus so is $w$.
\end{Rem}

\subsection{Convergence rate of IRMCMC}\label{sec:IRMCMCrate}
The following equivalent representation of Algorithm~\ref{algo:IRMCMC} is useful. 
Let $\{Z_n\}_{n\geq 0}$ be a sequence of independent and identically distributed random variables with $\proba(Z_1 = 1) = 1-\proba(Z_1 = 0) =  \epsilon$. Assume that $\{Z_n\}_{n\geq 0}$ and $\{Y_n\}_{n\geq 0}$ are independent and for each $n\geq 1$, $Z_n$ and $\F_{n-1}$ are independent.  Then, at round $n$, we can introduce $Z_n$, and write the conditional distribution of $X_n$ given $Z_n,\F_{n-1}$ as
\[
X_n\mid \F_{n-1},Z_n\sim\left\{
\begin{array}{l@{\mbox{ if }}l}
P(X_{n-1},\cdot) & Z_n = 0\\
\what\theta_{n-1}(\cdot) & Z_n = 1\,.
\end{array}
\right.
\]
Define 
\equh\label{eq:tau}
\tau_0 =0, \tau_{i+1} = \min\{k>\tau_i: Z_k = 1\} \mand n^* = \max\{k: \tau_k\leq n\}\,.
\eque
Observe that at each time $\tau_k>0$, conditioning on $Y_0, Y_1,\dots,Y_{\tau_k-1}$, $X_{\tau_k}$ is sampled from $\what\theta_{\tau_k-1}$, independent of $X_0,\dots,X_{\tau_k-1}$. Furthermore, $Y_0,\dots,Y_n$ are independent from $\tau_1,\dots,\tau_{n^*}$. Therefore, we first focus on
\equh\label{eq:eta}
\eta_n\eqdef \proba(X_{n+1}\in\cdot\mid Z_{n+1}=1) = \esp\what\theta_{n}(\cdot)\,,n\in\N\,.
\eque
We first obtain a bound on the total variation distance $\nnTV{\eta_n-\pi}$. 
Recall that, given two probability distributions $\mu$ and $\nu$, the total variation distance $\nnTV{\mu-\nu}$ is defined by:
$\nnTV{\mu-\nu} = \frac12\sup_{|f|_\infty\leq 1}|\mu(f)-\nu(f)|$.
For convenience, write 
\equh\label{eq:Bn}
B_n\eqdef |w|_\infty\sup_{|f|_\infty\leq 1}\esp\what\pi_{Y,n}(\wb f) + |w|_\infty^2\sup_{|f|_\infty\leq 1}\esp\pp{\what\pi_{Y,n}(\wb f)}^2, n\in\N.
\eque
Recall that throughout we assume $|w|_\infty<\infty$. 
\begin{Lem}\label{lem:etak} For all $n\in\N$, 
$\nnTV{\eta_n-\pi}
\leq B_n$.
\end{Lem}

The proof of Lemma~\ref{lem:etak} is postponed to next subsection. Lemma~\ref{lem:etak} yields an upper bound on the convergence rate of $\calL_{X_n}\weakto \pi$, as shown in the following theorem.  We set $B_0 = B_{-1} = 1$. 
\begin{Thm}\label{thm:IRMCMC}
Consider $\indn X$ generated from Algorithm~\ref{algo:IRMCMC}. 
Then, 
\equh\label{eq:boundXn}
\nnTV{\calL_{X_n}-\pi} 
\leq \sum_{\ell=0}^n (1-\epsilon)^{n-\ell}B_{\ell-1}.
\eque
Furthermore, for any bounded measurable function $f$, 
\begin{multline}\label{eq:L2fX}
\PE\bb{\frac1{\sqrt n}\summ i1n(f(X_i)-\pi(f))}^2\\
\leq \frac{80\epsilon^{-2}|f|_\infty^2}n + 64\epsilon^{-2}|f|_\infty^2 + |f|_\infty^2\pp{\frac1{\sqrt n}\summ k0{n-1}\sqrt{B_k}}^2, n\in\N.
\end{multline}

\end{Thm}
The proof of Theorem~\ref{thm:IRMCMC} is postponed to next subsection. 
\begin{Rem}
In  Theorem~\ref{thm:IRMCMC}, we do not assume  any ergodicity assumption on the kernel $P$. In the case $P$ is geometrically ergodic, one can improve (\ref{eq:boundXn}) quantitatively by bounding the term $\nnTV{\eta_k P^{n-k}-\pi}$ more effectively. For example, if $P$ is uniformly ergodic with rate $\rho$, then (\ref{eq:boundXn}) would become 
$\nnTV{\calL_{X_n}-\pi} 
\leq \sum_{\ell=0}^n \left[\rho(1-\epsilon)\right]^{n-\ell}B_{\ell-1}$.
A similar improvement can be formulated for (\ref{eq:L2fX}). However, these improvements do not change the rate but only the constant in the corollary below. Beside, such improvements will not be easily available if $P$ is sub-geometrically ergodic.
\end{Rem}
Now, as a corollary we obtain an upper bound on the convergence rate of IRMCMC algorithm, under the following assumption.
\assumpH
\item \label{A:Y}
There exist a finite constant $C$ such that for all measurable function $f:\;\calX\to\rset$,  with $|f|_\infty\leq 1$, 
\equh\label{eq:AY}
\esp \what\pi_{Y,n}(\wb f)\leq \frac {C}n \qmand \PE\bbpp{\what\pi_{Y,n}(\wb f)}^2 \leq \frac{C}n .
\eque
\assumpE
\begin{Rem}
Since $\PE\pp{\what\pi_{Y,n}(\wb f)}^2=n\inv\PE\left[n^{-1/2}\sum_{i=1}^n (f(Y_i)-\pi_Y(f))\right]^2$, the second part of Assumption \ref{A:Y} simply requires the finiteness of asymptotic variance under $\indn Y$ which is also a very desirable property in practice.  This is a fairly mild assumption that holds for many  processes with short-range dependence. See for example~\citet{haggstrom07variance} for further discussion when $\indn Y$ is a Markov chain.

The first part of \ref{A:Y} is also a fairly mild ergodicity assumption. 
\end{Rem}

\begin{Coro}\label{coro:1}
Consider the importance resampling MCMC (Algorithm~\ref{algo:IRMCMC}). If Assumption~\ref{A:Y} holds, then there exists a finite constant $C$ such that
\[
\nnTV{\calL_{X_n}-\pi} \leq \frac {C}n\,.
\]
Furthermore for any bounded measurable function $f$,
\[
\PE\bb{\frac1{\sqrt n}\summ i1n(f(X_i)-\pi(f))}^2 \leq C |f|_\infty^2\,, n\in\N\,.
\]
\end{Coro}
\begin{proof}Under Assumption~\ref{A:Y},~\eqref{eq:boundXn} yields
\eqnh
\nnTV{\calL_{X_n}-\pi} 
& \leq &  \frac Cn\bb{\summ \ell1{\floor{n/2}}(1-\epsilon)^{n-\ell}\frac n{\ell} + \summ\ell{\floor{n/2}+1}n(1-\epsilon)^{n-\ell}\frac n{\ell}} \\
& \leq & \frac Cn\bb{(1-\epsilon)^{n/2}n + \frac2{1-\epsilon}}.
\eqne
This proves the first conclusion. The proof of the second is staight-forward and thus omitted.
\end{proof}

\subsection{Proofs of Lemma~\ref{lem:etak} and Theorem~\ref{thm:IRMCMC}}
\begin{proof}[Proof of Lemma~\ref{lem:etak}]

Rewrite $\eta_n(f)$ as,
\eqnh
\eta_n(f) & = & \esp\bbpp{\summ j1{n}\frac{w(Y_j)}{\summ l1{n}w(Y_l)}f(Y_j)} \\
& = & \esp\bb{\frac1n\summ j1{n}w(Y_j)f(Y_j)+ \bbpp{1-\frac1n{\summ j1{n}w(Y_j)}}\summ j1{n}\frac{w(Y_j)f(Y_j)}{\summ l1{n}w(Y_l)}}\\
& = & \esp\bb{\what\pi_{Y,n}(wf) - \what\pi_{Y,n}(\wb w)\what\theta_n(f)},
\eqne
where in the third equality above we used the fact that $\pi_Y(w)=1$.
Since $\pi(f) = \pi_Y(wf)$, $  \nnTV{\eta_n -\pi} =  \sup_{|f|_\infty\leq 1}\frac12\bbpp{\eta_n(f)-\pi(f)}$ is bounded by
\begin{multline*}
 \frac12\sup_{|f|_\infty\leq 1} {\esp\what\pi_{Y,n}(\wb{wf})} + \frac12\sup_{|f|_\infty\leq 1}{\esp\pp{\what\pi_{Y,n}(\wb w)\what\theta_n(f)}}\\
\leq   \frac12\sup_{|f|_\infty\leq 1}{\esp\what\pi_{Y,n}(\wb{wf})}   + \frac12\sup_{|f|_\infty\leq 1}{\esp\bb{\what\pi_{Y,n}(\wb w)\pi_Y(wf)}}  \\
 + \frac12\sup_{|f|_\infty\leq 1}{\esp\bb{\what\pi_{Y,n}(\wb w)\pp{\what\theta_n(f)-\pi_Y(wf)}}}.
\end{multline*}
Observe that $\sup_{|f|_\infty\leq 1}\esp\spp{\what\pi_{Y,n}(\wb w)\pi_Y(wf)} = \sup_{|f|_\infty\leq 1}\pi(f)\esp\what\pi_{Y,n}(\wb w)
\leq |w|_\infty\sup_{|f|_\infty\leq 1}\what\pi_{Y,n}(\wb f)$
and $|w|_\infty\geq 1$.
Therefore,
\begin{multline}\label{eq:etan}
\nnTV{\eta_n-\pi}\\
\leq |w|_\infty{\sup_{|f|_\infty\leq 1}{\esp\what\pi_{Y,n}(\wb f)} + \frac12\sup_{|f|_\infty\leq 1}{\esp\bb{\what\pi_{Y,n}(\wb w)\bbpp{\what\theta_n(f)-\pi_Y(wf)}}}}.
\end{multline}
By Cauchy--Schwarz inequality, 
\begin{multline}
\sup_{|f|_\infty\leq 1}\esp\bb{{\what \pi_{Y,n}(\wb w)}\bbpp{\what\theta_n(f)-\pi_Y(wf)}}\\
\leq \bb{\esp\bbpp{\what \pi_{Y,n}(\wb w)}^2}^{1/2}\times \sup_{|f|_\infty\leq 1}\bb{\esp\bbpp{\what\theta_n(f)-\pi_Y(wf)}^2}^{1/2}\,.\label{eq:etan2}
\end{multline}
The first term is bounded by $|w|_\infty\sup_{|f|_\infty\leq 1}\sbb{\esp\spp{\what \pi_{Y,n}(\wb f)}^2}^{1/2}$. For the second term, observe that
\begin{multline}\label{eq:etan3}
\esp\bbpp{\what\theta_n(f)-\pi_Y(wf)}^2 \\
\leq 2\esp \bbpp{\what\theta_n(f)-\what\pi_{Y,n}(wf)}^2 + 2 \esp \bbpp{\what\pi_{Y,n}(wf) - \pi_Y(wf)}^2,
\end{multline}
and
\begin{multline*}
\esp \bbpp{\what\theta_n(f)-\what\pi_{Y,n}(wf)}^2 = \esp \bbpp{\summ j1{n}\frac{w(Y_j)f(Y_j)}{\summ l1{n}w(Y_l)} - \frac1n\summ j1{n}w(Y_j)f(Y_j)}^2\\
= \esp\bb{\bbpp{1-\what\pi_{Y,n}(w)}^2\what\theta_n^2(f)}
\leq \esp\bbpp{\pi_Y(w)-\what\pi_{Y,n}(w)}^2\\
\leq |w|_\infty^2\sup_{|g|_\infty\leq 1}\esp\bbpp{\what \pi_{Y,n}(\wb g)}^2,
\end{multline*}
and the above calculation holds for all $f: |f|_\infty\leq 1$. So,~\eqref{eq:etan3} becomes
\equh\label{eq:Bk}
\sup_{|f|_\infty\leq 1}\esp\pp{\what\theta_n(f) - \pi_Y(wf)}^2\leq 4|w|_\infty^2\sup_{|f|_\infty\leq 1}\esp(\what \pi_{Y,n}(\wb f))^2.
\eque
 Combining~\eqref{eq:etan},~\eqref{eq:etan2} and the above inequality yields the desired result.
\end{proof}

\begin{proof}[Proof of Theorem~\ref{thm:IRMCMC}]
 We recall that $\tau_{n^*}$ is the last time $k$ before $n$ that the main chain is sampled from $\what\theta_{k-1}$. Now, we can write
\eqnhn
\nnTV{\calL_{X_n}-\pi} 
& = & \sup_{|f|_\infty\leq 1}\frac12\abs{\summ k0n\esp(f(X_n)\indd{\tau_{n^*} = k}) - \pi(f)}\nonumber\\
& = & \sup_{|f|_\infty\leq 1}\frac12\abs{\summ k0n\proba(\tau_{n^*} = k)\sbb{\esp(f(X_n)\mid \tau_{n^*} = k) - \pi(f)}}\nonumber.
\eqnen
Thus
\equh
\nnTV{\calL_{X_n}-\pi}
 \leq  \summ k0n\proba(\tau_{n^*} = k)\sup_{|f|_\infty\leq 1}\frac12|\esp(f(X_n)\mid \tau_{n^*} = k) - \pi(f)|.\label{eq:bound}
\eque
Observe that the conditional distribution of $X_n$ given that $\tau_{n^*} = k\geq 1$, is $\eta_{k-1}P^{n-k}$ (set $\eta_0=\delta_{Y_0}$). Then, 
\eqnh
\sup_{|f|_\infty\leq 1}\frac12|\esp(f(X_n)\mid\tau_{n^*} = k) - \pi(f)| & = &
\sup_{|f|_\infty\leq 1}\frac12|\eta_{k-1}P^{n-k}(f) - \pi(f)| \\
& = & \nnTV{\eta_{k-1}P^{n-k}-\pi}\,.
\eqne
By  the fact that $\pi P = \pi$, we have $\nnTV{\eta_{k-1} P^{n-k}-\pi}  \leq \nnTV{\eta_{k-1}-\pi}\leq B_{k-1}$, by Lemma \ref{lem:etak}. Also $\proba(\tau_{n^*} = k)=\epsilon(1-\epsilon)^{n-k}$ for $k=1,\dots,n$ and $\proba(\tau_{n^*} = 0) = (1-\epsilon)^n$. Thus,~\eqref{eq:bound} becomes \eqref{eq:boundXn}. 

To establish (\ref{eq:L2fX}), we show that the partial sum $\sum_{k=1}^n \left(f(X_k)-\pi(f)\right)$ admits a well behaved martingale approximation. For a probability measure $\theta$ on $\calX$, define
\[\pi_\theta(A)=\epsilon\sum_{j=0}^\infty(1-\epsilon)^j(\theta P^j)(A),\;\;A\in\calB.\]
Clearly, $\pi_\theta$ is a  probability measure on $(\calX,\calB)$, and one can verify that  $\pi_\theta P_\theta = \pi_\theta$, and moreover
that for any bounded measurable function $f$, and $n\geq 1$,
\begin{equation}\label{eq:ratePtheta1}
P_\theta^nf(x)-\pi_\theta(f)=(1-\epsilon)^nP^nf(x)-\epsilon\sum_{j=n}^\infty(1-\epsilon)^j(\theta P^j)f.\end{equation}
Indeed, the case $n=1$ follows from the definition of $P_\theta$ in~\eqref{eq:Ptheta}. For $n\geq 1$, by induction, $P_\theta^{n+1}f(x)-\pi_\theta(f) = P_\theta^n(P_\theta f)(x) - \pi_\theta(P_\theta f)$ equals
\begin{multline*}
(1-\epsilon)^{n+1}P^{n+1}f(x) + (1-\epsilon)^n\epsilon\theta f - \epsilon\sif jn(1-\epsilon)^j(\theta P^j)[(1-\epsilon)Pf + \epsilon\theta f]\\
=(1-\epsilon)^{n+1}P^{n+1}f(x)  + \epsilon \sif j{n+1}(1-\epsilon)^j(\theta P^j)f.
\end{multline*}

It then follows from~\eqref{eq:ratePtheta1} that
$\nnTV{P_\theta^n(x,\cdot)-\pi_\theta}\leq 2(1-\epsilon)^n$, and consequently the function 
\equh\label{eq:gtheta}
g_\theta(x)=\sum_{j=0}^\infty \left(P_\theta^jf(x)-\pi_\theta(f)\right),
\eque
is well-defined with $|g_\theta|_\infty\leq 2\epsilon^{-1}|f|_\infty$, and satisfies Poisson's equation
\begin{equation}\label{eq:PoissonPtheta}
g_\theta(x)-P_\theta g_\theta(x)=f(x)-\pi_\theta(f),\;\;\;x\in\calX.\end{equation}
In particular, we have $f(X_k)-\pi_{\what\theta_{k-1}}(f)=g_{\what\theta_{k-1}}(X_k)-P_{\what\theta_{k-1}}g_{\what\theta_{k-1}}(X_k)$, almost surely. Using this, we write:
\[
\sum_{k=1}^n \left(f(X_k)-\pi(f)\right) =  \sum_{k=1}^n \left(\pi_{\what\theta_{k-1}}(f)-\pi(f)\right)+ \sum_{k=1}^n \left(f(X_k)-\pi_{\what\theta_{k-1}}(f)\right)
\]
with
\eqnhn
\sum_{k=1}^n \pp{f(X_k)-\pi_{\what\theta_{k-1}}(f)} & = & \sum_{k=1}^n \left(g_{\what\theta_{k-1}}(X_k)-P_{\what\theta_{k-1}}g_{\what\theta_{k-1}}(X_{k-1})\right)\nonumber\\
& & +\sum_{k=1}^n \left(P_{\what\theta_{k-1}}g_{\what\theta_{k-1}}(X_{k-1})-P_{\what\theta_{k}}g_{\what\theta_{k}}(X_{k})\right) \nonumber\\
& & + \sum_{k=1}^n \left(P_{\what\theta_{k}}g_{\what\theta_{k}}(X_{k})-P_{\what\theta_{k-1}}g_{\what\theta_{k-1}}(X_{k})\right).\label{eq:3terms}
\eqnen

From the definition of $\pi_\theta$, notice that we can write 
\[\sum_{k=1}^n \left(\pi_{\what\theta_{k-1}}(f)-\pi(f)\right)=\sum_{k=1}^n \what\theta_{k-1}(f_\epsilon-\pi(f_\epsilon)),\]
where $f_\epsilon(x)=\epsilon\sum_{j=0}^\infty(1-\epsilon)^j P^j f(x)$. 
Thus, 
\begin{multline*}
\esp\bb{\summ k1n\pp{\pi_{\what\theta_{k-1}}(f)-\pi(f)}}^2 \\
\leq 
\pp{\summ k1n\pp{\esp\what\theta_{k-1}^2(f_\epsilon - \pi(f_\epsilon))}^{1/2}}^2
\leq |f|_\infty^2\pp{\summ k0{n-1}\sqrt{B_k}}^2,
\end{multline*}
where in the last equality, we use the fact that  $\sup_{|f|_\infty\leq 1}\PE\what\theta_k^2(f-\pi(f))\leq B_k$,
established in~\eqref{eq:Bk} in Lemma~\ref{lem:etak}.

We now bound the three sums on the right-hand side of~\eqref{eq:3terms}. By~\eqref{eq:ratePtheta1},~\eqref{eq:gtheta} and~\eqref{eq:PoissonPtheta},  for any probability measures $\theta,\theta'$ and $x\in\calX$,
\[P_\theta g_\theta(x)-P_{\theta'}g_{\theta'}(x)=\int(\theta'-\theta)(\d z)\left(\epsilon\sum_{j=0}^\infty j(1-\epsilon)^j P^jf(z)\right).\]
This implies that 
\begin{multline*}
\abs{\sum_{k=1}^n \left(P_{\what\theta_{k}}g_{\what\theta_{k}}(X_{k})-P_{\what\theta_{k-1}}g_{\what\theta_{k-1}}(X_{k})\right)} \\
= \abs{(\what\theta_0-\what\theta_n)\pp{\epsilon\sif j0j(1-\epsilon)^jP^jf}}\leq \frac{2(1-\epsilon)}\epsilon|f|_\infty.
\end{multline*}
Next, observe  
\begin{multline*}
\abs{\sum_{k=1}^n \left(P_{\what\theta_{k-1}}g_{\what\theta_{k-1}}(X_{k-1})-P_{\what\theta_{k}}g_{\what\theta_{k}}(X_{k})\right)} \\
= \abs{P_{\what\theta_0}g_{\what\theta_0}(X_0) - P_{\what\theta_n}g_{\what\theta_n}(X_n)}\leq |g_{\what\theta_0}|_\infty + |g_{\what\theta_n}|_\infty\leq 4\epsilon\inv|f|_\infty.
\end{multline*}

Finally we also notice that $\sum_{k=1}^n \left(g_{\what\theta_{k-1}}(X_k)-P_{\what\theta_{k-1}}g_{\what\theta_{k-1}}(X_{k-1})\right) =:\summ k1n D_k$ is a martingale with respect to $\{\F_n\}$, whence
$\esp\pp{\summ k1nD_k}^2  = \summ k1n\esp D_k^2\leq 4n\sup_{\theta}|g_\theta|_\infty^2\leq 16\epsilon^{-2}|f|_\infty^2n$.
 Using all the above, we obtain~\eqref{eq:L2fX}.
\end{proof}

\subsection{An example on the lower bound}
We provide an example where $O(n\inv)$ is also the lower bound of the rate for both $\nnTV{\eta_n-\pi}$ and $\nnTV{\calL_{X_n}-\pi}$. This shows that the rate in our upper bound in Corollary~\ref{coro:1} is optimal. 
\begin{Example}\label{rem:2state}
Consider the simple case when $\calX = \{\pm1\}$, and $\pi = \pi_Y$. In this case, the weight function is uniform ($w \equiv 1$). Suppose the auxiliary chain $\{Y_n\}_{n\geq 0}$ has transition matrix
\[
P_Y = \bbpp{
\begin{array}{cc}
1-a & a\\
b & 1-b
\end{array}
}, \mwith a,b\in(0,1)\,.
\]
The corresponding Markov chain has stationary distribution $\pi_Y = (a+b)\inv(b,a)$ and eigenvalues $\lambda_1 = 1,\lambda_2 = 1-a-b$. Suppose $a+b\neq 1$ and the chain starts at $Y_0 = -1$. By straight-forward calculation, $\proba(Y_n = -1) = a/(a+b) + b/(a+b)\lambda_2^n$, and 
\[
\esp\what\pi_{Y,n}(\{-1\}) - \pi_Y(\{-1\})  = \frac a{a+b}\frac1n\frac{\lambda_2-\lambda_2^{n+1}}{1-\lambda_2}\,.
\]
It then follows from the definition that $\nnTV{\eta_n-\pi}\geq C/n$.

Furthermore, in~\eqref{eq:Ptheta} set $P(x,\cdot) = \pi(\cdot)$. That is, $P$ is the {\it best} kernel we can put into the algorithm, in the sense that it takes one step to arrive at the stationary distribution (although this is too ideal to be practical). Now,
\eqnh
\proba(X_n = -1) - \pi(\{-1\}) & = & (1-\epsilon)\pi(\{-1\}) + \epsilon\esp\what\pi_{Y,n}(\{-1\}) - \pi(\{-1\})\\
& = & \epsilon\bbpp{\esp\what\pi_{Y,n}(\{-1\}) - \pi_Y(\{-1\})}\,.
\eqne
It then follows that $\nnTV{\calL_{X_n}-\pi}\geq C/n$.

\end{Example}

\subsection{Multiple IRMCMC}\label{sec:mIRMCMC}
We discuss a multiple chain importance-resampling MCMC algorithm and establish a similar convergence rate as in Section~\ref{sec:IRMCMCrate}, by a repeated application of Theorem \ref{thm:IRMCMC}. For $m\geq 1$, and $\ell\in\{0,\ldots,m\}$, let $\pi^{(\ell)}$ be a probability measure on $\calX$, and $P_\ell$ a Markov kernel with invariant distribution $\pi^{(\ell)}$, such that $\pi^{(m)}=\pi$.
\begin{Algo}[Multiple IRMCMC]\label{algo:mIRMCMC}
Fix $\epsilon\in(0,1)$, and choose $(X_0^{(0)},\ldots,X^{(m)}_0)=(x_0^{(0)},\ldots,x^{(m)}_0)$.  Given $\F_n=\sigma\left\{(X_k^{(0)},\ldots,X^{(m)}_k),\;0\leq k\leq n\right\}$:
sample independently sample $X_{n+1}^{(0)}\sim P_0(X_n^{(0)},\cdot)$, and for $1\leq \ell\leq m$, $X_{n+1}\topp\ell\sim P_{\ell,\what\theta\topp{\ell-1}_{n}}(X_n\topp \ell,\cdot)$ with
\[
P_{\ell,\theta}(x,\cdot) = (1-\epsilon) P_\ell(x,\cdot) + \epsilon\theta(\cdot) 
\]
and
\[
\what\theta\topp{\ell-1}_n(\cdot) = \summ i1{n}\frac{w_\ell(X_i\topp{\ell-1})}{\summ j1{n}w_\ell(X_j\topp{\ell-1})}\delta_{X_i\topp{\ell-1}}(\cdot),
\]
with $w_\ell(x) = {\pi_\ell(x)}/{\pi_{\ell-1}(x)}, x\in\calX$.
\end{Algo}
To bound $\snnTV{\calL_{X_n\topp\ell}-\pi_\ell}$, it suffices to control 
\equh\label{eq:Bnl}
B_n\topp{\ell-1} \eqdef \sup_{|f|_\infty\leq 1}\esp\what\pi_{X\topp{\ell-1},n}(\wb f) + \sup_{|f|_\infty\leq 1}\esp\pp{\what\pi_{X\topp{\ell-1},n}(\wb f)}^2, n\in\N,
\eque
where this time $\what\pi_{X\topp\ell,n}(\wb f) \eqdef \what\pi_{X\topp\ell,n}(f) - \pi_{\ell}(f)$.
In fact, it suffices to control $B_n^{(0)}$, which is the purpose of the following assumption. 

\assumpH
\item \label{A:mIRMCMC}
As $n\to\infty$, the initial Markov chain $\{X_n\topp 0\}_{n\geq 0}$ satisfies $B_n^{(0)}\leq C/n$.
\assumpE

\begin{Thm}\label{thm:mIRMCMC}
Consider the multiple IRMCMC (Algorithm~\ref{algo:mIRMCMC}) for which Assumption~\ref{A:mIRMCMC} holds and $\max_{\ell=1,\dots,m}|w_\ell|_\infty<\infty$. Then for $\ell = 1,\dots,m$, there exists a finite constant $C$ such that for $n\geq 2$,
\equh\label{eq:lognm}
\nnTV{\calL_{X_n\topp \ell} - \pi_\ell} \leq \frac{C(\log n)^{\ell-1}}n\,,
\eque
and for any bounded measurable function $f$,
\begin{equation}\label{eq:varm}
\esp\bb{\frac1{\sqrt n}\summ i1n\pp{f(X_i\topp{\ell})-\pi_\ell(f)}}^2 \leq C.\end{equation}

\end{Thm}
\begin{proof}
This follows easily from a repeated application of Theorem \ref{thm:IRMCMC}.
\end{proof}


\nnewpage
\section{Interacting tempering algorithm}\label{sec:EE}
In this section, we consider the interacting tempering algorithm as follows. Recall that the auxiliary chain $\{Y_n\}_{n\geq 0}$ evolves independently from the main chain $\{X_n\}_{n\geq 0}$. 
\begin{Algo}[Interacting Tempering Algorithm]\label{algo:EE}
Fix $\epsilon\in(0,1)$. Start $X_0 = x_0$ and $Y_0 = y_0$. At each round $n$,  generate
\[
X_n \sim\left\{
\begin{array}{ll}
P(X_{n-1},\cdot) & \mbox{ w.p.~} 1-\epsilon\\
K_{\what\pi_{Y,n-1}}(X_{n-1},\cdot) & \mbox{ w.p.~} \epsilon
\end{array}
\right.\,,
\]
where $\what\theta_n = \what\pi_{Y,n}$ is the empirical measure associated to $\{Y_n\}_{n\geq 0}$ and $K_\theta$ is defined by 
\[
K_\theta(x,A)=\textbf{1}_A(x) + \int_{\calX}\left(1\wedge\frac{\pi(z)\pi_Y(x)}{\pi(x)\pi_Y(z)}\right)\left(\textbf{1}_A(z)-\textbf{1}_A(x)\right)\theta(\d z).
\]
In other words, 
for all non-negative functions $h:\calX\to\R$ and $n\in\N$,
\begin{equation}\label{dynEE}
\PE_{x}\left(h(X_{n+1})\mid \F_{n}\right)=P_{\what\pi_{Y,n}}h(X_{n}) \;\;\mbox{ almost surely,}
\end{equation}
where for any probability measure $\theta$ on $\calX$, $P_\theta$ is defined as 
\equh\label{eq:EE}
P_\theta(x,A)=(1-\epsilon)P(x,A)+\epsilon K_\theta(x,A),
\eque
Recall that we write $\pi(\d x) \equiv \pi(x)\d x$ and similarly for $\pi_Y$ with a little abuse of language, and $w(x) = \pi(x)/\pi_Y(x)$. We assume $|w|_\infty<\infty$. 
\end{Algo}

The kernel $K_{\pi_Y}$ is the Independent Metropolis kernel with target $\pi$ and proposal $\pi_Y$. It is well known that under the assumption $|w|_\infty<\infty$ (recall Remark~\ref{rem:w}), the kernel $K_{\pi_Y}$ is uniformly ergodic \citep{mengersen96rates}, and this property is inherited by $P_{\pi_Y}$. That is, there exist $C_0<\infty,\; \rho\in (0,1)$, such that
\begin{equation}\label{eq:rateconv}
  \nnTV{P^n_{\pi_Y}(x,\cdot) - \pi(\cdot)}\leq C_0 \rho^n,\;\;\;n\geq 0.
\end{equation}

\subsection{Convergence rate of IT algorithm}\label{sec:EErate}
We make the following assumptions.

\assumpH
\item \label{A:subGaussian}
There exist a finite universal  constant $C$ such that for any measurable function $f:\;\calX\to\rset$,  with $|f|_\infty\leq 1$, 
\[
\sup_{n}\PP\left(\left|\frac{1}{\sqrt{n}}\summ j1n \pp{f(Y_j) - \pi_Y(f)}\right|>x\right)  \leq C\exp\left(-\frac{x^2}{C\sigma^2(f)}\right),
\]
where $\sigma^2(f)\eqdef \var_{\pi_Y}(f)$.
\assumpE

\assumpH
\item \label{A:w} The function $w:\;\calX\to\rset$ is continuous (with respect to the metric on $\calX$),  and
\equh\label{eq:w}
\sup_{x\in\calX}\frac{\phi(x)}{w^2(x)}<\infty,
\eque
where $\phi(x)\eqdef \pi_Y\left(\{z:\; w(z)\leq w(x)\}\right)$.
\assumpE

\assumpH
\item \label{A:continuous} The kernel $P$ is such that if $f:\calX\to\rset$ is continuous, then $Pf$ is also continuous.
\assumpE

\begin{Rem} The deviation bound \ref{A:subGaussian} appears naturally
in the proof although this type of bounds are not widely available for
Markov chains. A continuous time version appeared in  \citet[Proposition
1.2]{cattiaux08deviation} but extension to discrete time Markov chains
along the same arguments is apparently not straightforward.
\end{Rem}
\begin{Rem}
Assumption~\ref{A:w} can be difficult to check in practice, but is not overly restrictive. For example, consider $\calX = \R$ and $\pi_Y = \pi^T$ with some $T\in(0,1)$. For the sake of simplicity, we focus on $x\in\R_+$ and define $\phi_+(x) \eqdef \pi_Y(\{z>0:w(z)\leq w(x)\})$. Suppose the density $\pi(x)$ decays asymptotically as $x^{-\alpha}$ for $\alpha>1$ as $x\to\infty$. Then, $\pi_Y(x)\sim x^{-T\alpha}$ and $w(x)\sim x^{(T-1)\alpha}$. Here and below, we write $a(x)\sim b(x)$ if $\lim_{x\to\infty}a(x)/b(x) = 1$.  Assume further that $T\alpha>1$. Then, $\phi_+(x)\sim (T\alpha-1)\inv x^{1-T\alpha}$ and 
$\frac{\phi_+(x)}{w^2(x)} \sim\frac1{T\alpha-1}x^{1+2\alpha-3T\alpha}$.
Therefore,~\eqref{eq:w} holds, if $T>(1+2\alpha)/(3\alpha)$. 
\end{Rem}

\begin{Thm}\label{thm:EE}
Consider the IT algorithm described as above and suppose that Assumptions \ref{A:subGaussian}--\ref{A:continuous} hold. Then, there exists a constant $C$, such that for all continuous bounded functions $f:\calX\to\R$ and $n\in\N$,
\equh\label{eq:thm:EE}
\left|\esp\left(f(X_n)-\pi(f)\right)\right|\leq \frac{C|f|_\infty}{\sqrt{n}}.
\eque
\end{Thm}
\begin{proof}
Fix $n\geq 2$ and $1\leq q\leq n$. Fix $f:\;\calX\to\rset$ with $|f|_\infty= 1$. Then write
\begin{multline*}
\PE_xf(X_n)-P_{\pi_Y}^nf(x)\\
=\PE_x\left(P_{\pi_Y}^{n-q}f(X_q)-P_{\pi_Y}^nf(x)\right)-\PE_x\left(P_{\pi_Y}^{n-q}f(X_q)-f(X_n)\right).
\end{multline*}
For the first term we can use \eqref{eq:rateconv} to get:
$|\PE_x(P_{\pi_Y}^{n-q}f(X_q)-P_{\pi_Y}^nf(x))|\leq C \rho^{n-q}$,
for some finite constant $C$ that does not depend on $f$. For the second term, we write:
\eqnhn
& & \PE_x\left(P_{\pi_Y}^{n-q}f(X_q)-f(X_n)\right)\nonumber\\
& = & \PE_x\left[\sum_{j=q}^{n-1}\left(P_{\pi_Y}^{n-j}f(X_j)-P_{\pi_Y}^{n-j-1}f(X_{j+1})\right)\right]\nonumber\\
& = & \sum_{j=q}^{n-1}\PE_x\left[P_{\pi_Y}^{n-j}f(X_j)-\PE_x\left(P_{\pi_Y}^{n-j-1}f(X_{j+1})\mid \F_j\right)\right]\nonumber\\
& = & \sum_{j=q}^{n-1}\PE_x\left[P_{\pi_Y}^{n-j}f(X_j)-P_{\what\pi_{Y,j}}P_{\pi_Y}^{n-j-1}f(X_{j})\right]\nonumber\\
& = &\sum_{j=q}^{n-1}C_0 \rho^{n-j-1}\PE_x\left[\left(P_{\pi_Y}-P_{\what\pi_{Y,j}}\right)\zeta_{n,j}(X_j)\right],\label{eq:Ex}
\eqnen
where in the last line we write 
\[
\zeta_{n,j}(x)=\frac{P_{\pi_Y}^{n-j-1}(f(x)-\pi_Y(f))}{C_0\rho^{n-j-1}},\;\;x\in\calX\,,
\]
with $C_0$ and $\rho$ chosen as in~\eqref{eq:rateconv}.
As a consequence of~\eqref{eq:rateconv}, $|\zeta_{n,j}|_\infty\leq 1$. It is also continuous by the continuity of $f$ and Assumption~\ref{A:continuous}.

To simplify the notation, for any function $g:\;\calX\to \rset$, define

\begin{equation}\label{eq:Hg}
H_g(x,z)\eqdef\alpha(x,z)\left(g(z)-g(x)\right),\quad x,z\in\calX,\end{equation}
where 
\equh\label{eq:alpha}
\alpha(x,z)\eqdef 1\wedge \frac{w(z)}{w(x)}.
\eque
 Thus, we can write
\[
P_\theta g(x)-P_{\pi_Y}g(x)=\epsilon\int H_g(x,z)(\theta(\d z)-\pi_Y(\d z)).
\]
For any $g:\calX\to\R$, we introduce the class of functions
$\filF_g\eqdef\ccbb{z\mapsto H_g(x,z):\;x\in\calX}$,
and the empirical process 
\[\G_n(h)\eqdef \frac{1}{\sqrt{n}}\sum_{j=1}^n \left(h(Y_j)-\pi_Y(h)\right),\;\;\; h\in\filF_g.
\]
Therefore, the expectation term in~\eqref{eq:Ex} becomes
\begin{multline*}
 \esp_x\bb{\bbpp{P_{\pi_Y}-P_{\what\pi_{Y,j}}}\zeta_{n,j}(X_j)}
= \epsilon\esp_x\bb{\int H_{\zeta_{n,j}}(X_j,z)(\pi_Y(\d z)-\what\pi_{Y,j}(\d z))}\\
= -\epsilon\esp_x\bb{\frac1j\summ \ell1j H_{\zeta_{n,j}}(X_j,Y_\ell) - \int_\calX H_{\zeta_{n,j}}(X_j,z)\pi_Y(\d z)}\\
=  -\frac{\epsilon}{\sqrt{j}}\esp_x\bb{\G_j\bbpp{H_{\zeta_{n,j}}(X_j,\cdot)}}\,,
\end{multline*}
whence
\eqnh
\left|\PE_x\left(P_{\pi_Y}^{n-q}f(X_q)-f(X_n)\right)\right| & = & \abs{\epsilon\sum_{j=q}^{n-1}\frac{C_0\rho^{n-j-1}}{\sqrt{j}}\esp_x\bb{\G_j\bbpp{H_{\zeta_{n,j}}(X_j,\cdot)}}}\\
 & \leq & C_0\sum_{j=q}^{n-1}\frac{\rho^{n-j-1}}{\sqrt{j}}\PE_x\left(\sup_{h\in \F_{\zeta_{n,j}}}\left|\G_j(h)\right|\right).
\eqne
We prove in Lemma \ref{lem2} below that for any continuous function $g:\calX\to\rset$ such that $|g|_\infty\leq 1$, 
$\PE_x\left(\sup_{h\in \F_{g}}\left|\G_n(h)\right|\right)\leq C$,
for some constant $C$ that does not depend on $n$ nor $g$. We conclude that
$|\PE_x(P_{\pi_Y}^{n-q}f(X_q)-f(X_n))|\leq C\sum_{j=q}^{n-1}\frac{1}{\sqrt{j}}\rho^{n-j-1}$.
Thus for any $1\leq q\leq n$,
\[
\left|\PE_x\left(f(X_n)\right)-\pi_Y(f)\right|\leq C \left\{\rho^n + \rho^{n-q} + \epsilon \sum_{j=q}^{n-1}\frac{\rho^{n-j-1}}{\sqrt{j}}\right\}\leq Cn^{-1/2},
\]
by choosing $q=n-\lfloor\frac{-\log n}{2\log\rho}\rfloor$.
\end{proof}

We rely on the following technical result on the auxiliary chain $\{Y_n\}_{n\geq 0}$. 

\begin{Lem}\label{lem2}Suppose that Assumptions \ref{A:subGaussian} and \ref{A:w} hold. Then there exists a constant $C$ such that for all continuous function $g:\calX\to\rset$ such that $|g|_\infty\leq 1$,
\[
\sup_{n\in\N}\PE_x\left(\sup_{h\in \F_g}\left|\G_n(h)\right|\right)\leq C.
\]
\end{Lem}
\begin{proof}
Throughout the proof $n\geq 1$ is fixed. Assumption \ref{A:subGaussian} suggests the following metric on $\F_g$: 
\[\textsf{d}(h_1,h_2)=\sigma(h_1-h_2)=\left(\int_\calX\left(h_1(x)-h_2(x)\right)^2\pi_Y(\d x)\right)^{1/2},\]
which has the following properties.
For  $x_1,x_2\in\calX$, it is easy to check that 
\begin{equation}\label{dist1}
\left |H_g(x_1,z)-H_g(x_2,z)\right|\leq 2\left|\alpha(x_1,z)-\alpha(x_2,z)\right| + \left|g(x_1)-g(x_2)\right|.\end{equation}
It follows that
\begin{multline}\label{dist2}
\textsf{d}\left(H_g(x_1,\cdot),H_g(x_2,\cdot)\right)\\
\leq \sqrt 2\left|g(x_1)-g(x_2)\right| 
+ 2\sqrt2\sqrt{\int\left|\alpha(x_1,z)-\alpha(x_2,z)\right|^2\pi_Y(\d z)}.\end{multline}
This implies that the diameter of $\F_g$ is bounded by $\delta(\F_g)=4\sqrt{2}$.  It also implies that with respect to $\textsf{d}$, the empirical process $\{\G_n(h),\;h\in \F_g\}$ is separable. Indeed, for $x\in\calX$ arbitrary and $h=H_g(x,\cdot)$, using the Polish assumption, we can find a sequence $x_m\in\calX$ ($x_m$ belongs to a countable subset of $\calX$) such that $x_m\to x$, as $m\to\infty$. Setting $h_m=H_g(x_m,\cdot)$, it follows from (\ref{dist2}) and the continuity of $g$ and $w$ that $h_m\to h$ in $(\F_g,\textsf{d})$, and (\ref{dist1}) easily show that $\G_n(h_m)-\G_n(h)=n^{-1/2}\sum_{\ell=1}^n\left(H_g(x,Y_\ell)-H_g(x_m,Y_\ell)\right)+\sqrt{n}\pi_Y\left(H_g(x,\cdot)-H_g(x_m,\cdot)\right)\to 0$ as $m\to\infty$ for all realizations of $\{Y_1,\ldots,Y_n\}$.

For any  $h_1,h_2\in\F_{g}$, Assumption \ref{A:subGaussian} implies that for any $x>0$
\[
\PP_x\left(\left|\G_n(h_1)-\G_n(h_2)\right|>x\right)\leq C\exp\left(-\frac{x^2}{c\textsf{d}^2(h_1,h_2)}\right).
\]
Here, the constant $C$ above is universal for all $g$ such that $|g|_\infty\leq 1$: indeed,~\eqref{eq:Hg} implies that for such a function $g$, $h\in\filF_g$ implies $|h|_\infty\leq 2$.
Then we apply \citet[Corollary 2.2.8]{vandervaart96weak} to conclude that for $h_{0,g}\in\F_g$, there exists a constant $C$ independent of $g$, such that
\[
\PE_x\left(\sup_{h\in \F_{g}}\left|\G_n(h)\right|\right)\leq \PE_x|\G_n(h_{0,g})| + C\int_0^{\delta(\F_g)} \sqrt{1+\log\textsf{D}(\epsilon,\F_g,\textsf{d})}\;\d\epsilon<\infty,\]
where $\textsf{D}(\epsilon,\F_g,\textsf{d})$ is the packing number of $\F_g$ with respect to $\textsf{d}$. Since all elements of $\F_g$ have a sup-norm of at most two, Assumption \ref{A:subGaussian} implies that $\sup_{n\in\N}\PE_x|\G_n(h_{0,g})|\leq C<\infty$, where $C$ does not depend on $g$. To control the entropy number, we further bound the right hand of (\ref{dist2}).

Without loss of generality, assume $x_1,x_2\in\calX$ and $w(x_1)<w(x_2)$.  If $w(x_1)\vee w(x_2)\leq w(z)$, then $\alpha(x_1,z)-\alpha(x_2,z)=0$.  If $w(z)\leq w(x_1)$, then 
\[
\left|\alpha(x_1,z)-\alpha(x_2,z)\right|^2=\left|\frac{w(z)}{w(x_1)} - \frac{w(z)}{w(x_2)}\right|^2\leq \frac{1}{w(x_1)^2}\left(w(x_2)-w(x_1)\right)^2.
\]
If $w(x_1)\leq w(z)\leq w(x_2)$, then
\[\left|\alpha(x_1,z)-\alpha(x_2,z)\right|^2=\left|1- \frac{w(z)}{w(x_2)}\right|^2
\leq \frac{1}{w(x_2)^2}\left(w(x_2)-w(x_1)\right)^2.\]

Thus
\begin{multline*}
\int\left|\alpha(x_1,z)-\alpha(x_2,z)\right|^2\pi_Y(\d z)\\
\leq \pp{\frac{\phi(x_1)}{w(x_1)^2} 
 + \frac{\phi(x_2)}{w(x_2)^2}}\left(w(x_2)-w(x_1)\right)^2\leq C\left(w(x_2)-w(x_1)\right)^2,
\end{multline*}
where $\phi(x)\eqdef \pi_Y\left(\{z:\; w(z)\leq w(x)\}\right)$, and the last inequality follows from \ref{A:w}. 
Together with (\ref{dist2}), we conclude from this bound that  there exists a constant $C_0$ independent of $g$ such that
\begin{equation}\label{lip}
\textsf{d}\left(H_g(x_1,\cdot),H_g(x_2,\cdot)\right)
\leq C_0 (\left|g(x_1)-g(x_2)\right|+ \left|w(x_2)-w(x_1)\right|).
\end{equation}
Since $|g|_\infty\leq 1$ and $w(x)\in [0,|w|_\infty]$, this implies that the $\epsilon$-packing number of $(\F_g,\textsf{d})$ is at most of order  $\epsilon^{-2}$, independent of $g$. A detailed proof is provided below. It then follows that $\int_0^{\delta(\F_g)} \sqrt{1+\log\textsf{D}(\epsilon,\F_g,\textsf{d})}\;\d\epsilon\leq C \int_0^{\delta(\F_g)}\sqrt{1+\log(1/\epsilon)}d\epsilon<\infty$, which proves the lemma.

To complete the proof we show that the $\epsilon$-packing number of $(\filF_g,\textsf{d})$ is at most of order $\epsilon^{-2}$, independent of $g$. That is, the cardinality of any {\it $\epsilon$-separate set} is at most of order $\epsilon^{-2}$ (recall that a set is an $\epsilon$-separate set if any two points of this set have distance larger than $\epsilon$). Notice that the functions in $\filF_g$ are indexed by $x\in\calX$.

First, one can divide the set $\calX$ into $N = \floor{2/\epsilon}+1$ disjoint subsets $S(1),\dots,S(N)$, so that for every two points $x,y$ within the same $S(i)$, $|g(x)-g(y)|<\epsilon$. Notice that $N$ does not depend on $g$. 
For example, consider $g\inv([-1,-1+\epsilon]), g\inv((-1+\epsilon, -1+2\epsilon]), \dots, g\inv((-1+(N-1)\epsilon,1])$. 

Second, for each set $S(i)$, one can again divide it into $N' = \floor{|w|_\infty/\epsilon} + 1$ disjoint subsets, denoted by $S(i,j), j=1,\dots,N'$, so that within each $S(i,j)$, for every two points $x,y$, $|w(x)-w(y)|<\epsilon$.

Finally,  $\{S(i,j)\}_{i=1,\dots,N, j=1,\dots,N'}$ form a disjoint partition of $\calX$. The construction and~\eqref{lip} requires that any  $2C_0\epsilon$-separate set contains at most one point in each $S(i,j)$. Therefore, the $\epsilon$-packing number is at most of order $1/\epsilon^2$. 
\end{proof}

\bibliographystyle{apalike}
\bibliography{references}

\end{document}